\documentclass[a4paper,12pt]{article}
\usepackage{amsmath,amsthm,amssymb}

\usepackage[framemethod=TikZ]{mdframed}

\newtheorem{theorem}{Theorem}[section]

\newtheorem{corollary}[theorem]{Corollary}
\newtheorem{proposition}[theorem]{Proposition}

\theoremstyle{definition}
\newtheorem{definition}[theorem]{Definition}
\newtheorem{example}[theorem]{Example}
\newtheorem{remark}[theorem]{Remark}
\newtheorem*{problem}{Problem}

\numberwithin{equation}{section}

\newcommand{\abs}[1]{\lvert#1\rvert} 
\newcommand{\norm}[1]{\lVert#1\rVert} 
\newcommand{\dpair}[2]{\langle#1, #2\rangle} 

\newcommand{\R}{\mathbb{R}}  
\newcommand{\N}{\mathbb{N}}  
\newcommand{\F}{\mathcal{F}} 

\numberwithin{equation}{section}

\DeclareMathOperator{\h}{\mathbf{h}} 

\makeatletter
\let\@fnsymbol\@arabic
\makeatother

\author{
    Armando W. Guti\'{e}rrez\thanks{
        VTT Technical Research Centre of Finland Ltd, \texttt{armando.w.gutierrez@vtt.fi}} 
    \; and \; 
    Olavi Nevanlinna\thanks{
        Aalto University, 
        \texttt{olavi.nevanlinna@aalto.fi}}
    }

\title{Subinvariant metric functionals for 
nonexpansive mappings}
\date{}

\begin{document}
\maketitle

\begin{abstract}
We investigate the existence of subinvariant 
metric functionals for commuting families of 
nonexpansive mappings in noncompact subsets 
of Banach spaces. Our findings underscore 
the practicality of metric functionals when
searching for fixed points of nonexpansive 
mappings. To demonstrate 
this, we additionally investigate subsets of 
Banach spaces that have only nontrivial 
metric functionals. We particularly show that 
in certain cases every metric functional has 
a unique minimizer; thus, subinvariance 
implies the existence of a fixed point.
\end{abstract}

\section{Introduction}
The purpose of this note is to demonstrate 
the utility of metric functionals in the 
study of fixed points of nonexpansive mappings. 

A mapping $T:E \to E$ defined on a metric 
space $(E,d)$ is called nonexpansive if 
$\,d(Tx,Ty) \leq d(x,y)$ for all $x,y\in E$.
Nonexpansive mappings are not simply abstract 
objects studied by mathematicians; they also 
form the backbone of many important 
methods used by applied science
practitioners. For example, nonexpansive 
mappings are naturally found in the design of 
iterative methods for optimization 
algorithms, nonlinear evolution equations,
and control systems \cite{NAandOpt2010, FPalg2011}.
Such methods are designed as follows: one 
associates the original problem with a family 
of nonexpansive mappings $T_{\mu}:E\to E$ so
that a common fixed point $z=T_{\mu} z$ 
(for all $\mu$) gives a solution to the 
problem, then one builds in the metric 
space $E$ a sequence of points $(x_\mu)$ that
under certain conditions converges to $z$. 
It is worth mentioning that the existence of 
a common fixed point for a given family of 
nonexpansive mappings is not guaranteed in 
general.

The main key in our research approach is the 
notion of a metric functional. The idea is 
simple and of common form within mathematics: 
one considers the metric space been 
embedded into a larger space (the set of 
metric functionals) where a solution to a
``weaker" problem may be available, then 
sometimes one can show that this procedure
yields a solution in the original space and 
formulation. This idea is rigorously defined 
below.
 
\begin{definition}
Let $(E,d)$ be a metric space. We denote 
by $\R^E$ the space of all functionals from $E$ 
to $\R$ and equip it with the topology of 
pointwise convergence. We fix a point $o\in E$ 
and consider the mapping $w\mapsto \h_w$ 
from $E$ to $\R^E$ defined by the formula
\begin{equation}\label{eq:int_mf}
    \h_w(x) = d(x,w)-d(o,w)\ \ 
    \text{for all}\ \ x\in E. 
\end{equation}
We denote by $E^{\diamondsuit}$ the closure 
of the set
$\,\{\h_w \mid w\in E\}$ in $\R^E$ and we 
call each element of $E^{\diamondsuit}$ a
metric functional. 
\end{definition}

\begin{proposition}
\textup{(\cite[Chapter~3]{GutierrezThesis})}
The following properties hold:
\begin{enumerate}
\item The mapping $w\mapsto\h_w$ from $E$ to 
    $\R^E$ is injective and continuous.
\item The space $E^{\diamondsuit}$ is compact 
    and Hausdorff.     
\end{enumerate}
\end{proposition}

Since each $w\in E$ is uniquely identified 
with $\h_w$, we may view $E$ as been embedded 
into the compact space $E^\diamondsuit$. For
that reason, we call the metric functionals 
of the form (\ref{eq:int_mf}) \emph{internal}. 

When a subset $X$ of $E$ is considered 
equipped with the same metric $d$, we can 
similarly build the compact space 
$X^\diamondsuit$. It is worth noting that 
each internal metric functional 
$\h_w \in X^\diamondsuit$ can be trivially 
defined on the whole space $E$; thus, for 
each $\,\h\in X^\diamondsuit\,$ 
there exists 
$\,\widetilde{\h}\in E^\diamondsuit\,$ such 
that $\,\widetilde{\h}(x)=\h(x)$ for all 
$x\in X$. This extension statement
was recorded by Karlsson in 
\cite{Karlsson2021}. In this context,
{\it we will always assume that}
$X^\diamondsuit$ is a subset of 
$E^\diamondsuit$ and hence operates in all 
points of $E$. 

\begin{definition}  
We say that a functional $f \in \R^E\,$ is 
\emph{subinvariant} for a mapping 
$T: E \to E$ if for all $x \in E$ we have
$f(Tx) \leq f(x)$.
\end{definition}

It follows readily from the previous 
definitions that the existence of internal 
metric functionals which are subinvariant and 
the existence of fixed points coincide for 
nonexpansive mappings. We record a precise
statement below.

\begin{proposition}
\label{proppu1}  
Given a mapping $T: E \to E$, the 
following holds: 
\begin{enumerate}
\item If an internal metric functional 
    $\,\h_z \in E^\diamondsuit$ is subinvariant 
    for $T$ then $z\in E$ is a fixed point 
    of $T$.
\item If $T$ is nonexpansive and has a fixed 
    point $z\in E$ then the internal metric 
    functional $\,\h_z \in E^\diamondsuit$ is 
    subinvariant for $T$.
\end{enumerate}
\end{proposition}

We may face different situations with 
subinvariant metric functionals that are not 
internal. To begin with, some spaces have the 
metric functional that vanishes everywhere, 
and hence it is subinvariant for all mappings. 
Secondly, there may be subinvariant metric 
functionals, which are not internal but 
provide the existence of fixed points. And 
finally, if for some metric functional we 
would have $\h(Tx) < \h(x)$ for all $x\in E$, 
then $T$ cannot have a fixed point in $E$. 

\section{Problem statement}

In what follows we will consider  
nonexpansive mappings defined on Banach 
spaces or their subsets equipped with the 
same norm. To be precise, we want to study 
the following problem.

\begin{mdframed}
\begin{problem}
Let $X$ be a subset of a Banach space $E$ and 
let $\F$ be a commuting family of 
nonexpansive mappings from $X$ to itself. 
Find a metric functional 
$\h\in X^\diamondsuit$ that is subinvariant
for all $T\in\F$.    
\end{problem}  
\end{mdframed}

Next we present a simple example that fits 
our problem statement and is related to the 
applications mentioned in the Introduction.

\begin{example}
Let us assume that $A:E\to E$ is a 
linear operator and $b$ is some point in $E$. 
For $\mu\in \R$ let us consider the mapping 
$T_\mu : E\to E$ defined by the formula 
$T_\mu x = (1 - \mu)x + \mu(Ax + b)$ for all
$x \in E$. We can show that for all 
$\mu,\nu \in \R$ and for all $x \in E$ we have
$$
T_{\mu}T_{\nu}x = (1 - \mu)(1 - \nu)x 
    + (\nu - 2\mu \nu + \mu)Ax 
    + \mu \nu A(Ax + b) + (\nu - \mu \nu + \mu)b.
$$
Therefore, we have 
$T_\mu T_\nu = T_\nu T_\mu$. Now, if we 
assume that $\norm{Ax}\leq \norm{x}$ for all 
$ x\in E$ then 
$\F = \{ T_\mu\,\mid\,0 < \mu < 1 \}$ is a 
commuting family of nonexpansive mappings. 

By a result of O. Nevanlinna 
\cite[Proposition~1.4.2]{ONevanlinna1993},
we can extend the previous example as 
follows: let $\{q_s\}$ be a family of 
nonvanishing polynomials and set 
$$
p_s (\lambda) = 1 - (1 - \lambda)q_s (\lambda).
$$ 
Consider the affine mappings 
$T_s x = p_s(A) x + q_s(A)b$ where $A$ is a 
bounded linear operator. Then 
$\mathcal F=\{T_s\}$ is a commuting family. 
Depending on the spectrum $\sigma(A)$ it is 
possible that $\norm{p_s(A)} \leq 1$. 
Collecting all such pairs $q_s, p_s$ we have 
a commuting family of affine nonexpansive 
mappings.  
\end{example}

Before stating our main results, we want to
present below another example that motivates 
our study.  

\begin{example}\label{ex:1}
Let us assume that $E$ is any of the Banach 
spaces $c_0$, $\ell_1$, $\ell_2$, 
$\ell_\infty$ consisting of real sequences
$x = (x_k)_{k\geq 1}$. Let us consider the 
mapping $T: E\to E$ given by the formula
$$
T(x_1, x_2, x_3,\dots) 
    = (1, x_1, x_2, x_3,\dots ).
$$
It is clear that $T$ is both affine and 
isometric. From purely algebraic 
considerations, we observe that the only 
sequence which is mapped back to itself is 
$z=(1,1,1, \dots)$. Hence, $T$ has a unique 
fixed point in the space $E=\ell_\infty$ and 
$\h_z$ is a unique subinvariant metric 
functional. Now, $z$ is not in any of the 
smaller spaces, but we can construct 
subinvariant metric functionals for $T$ by 
considering the ``truncated" vectors 
$a_n=T^n0$ and taking limits of $\h_{a_n}$.
We will show in 
Proposition~\ref{prop:l1-l2claims} that if 
$E=\ell_2$ then the \emph{only} metric 
functional that is subinvariant for $T$ 
vanishes identically. Since there is no other 
subinvariant metric functional, $T$ has no 
fixed points in $\ell_2$. In the same 
proposition we will show also that if 
$E= \ell_1$ then there are infinitely many 
metric functionals that are subinvariant for 
$T$, all of them nontrivial. For example, the 
unbounded sequence $(a_n)$ in $\ell_1$ 
generates the metric functional
$$
\h(x) = \sum_{k=1}^\infty (\abs{x_k - 1} -1),  
$$ 
which is subinvariant for $T$. In fact, we 
have $\h(Tx) = \h(x) - 1$ for all 
$x\in\ell_1$, and hence $T$ has no fixed 
points in $\ell_1$. Finally, if $E=c_0$ then 
the sequence $(a_n)$ in $c_0$ generates the 
metric functional 
$$
\h(x) = \sup_{k\ge 1} \abs{x_k - 1} - 1,
$$ 
which is subinvariant for $T$.
\end{example}   
    
\begin{remark}
A theorem of Gaubert and Vigeral  
\cite{Gaubert-Vigeral-2012} implies the 
existence of a metric functional that is 
subinvariant for each element of the family 
$\F = \{ T^n \mid n\geq 1\}$, where $T:X\to X$ 
is nonexpansive and $X$ is a star-shaped 
subset of $E$. Karlsson showed a similar 
result in \cite[Proposition~13]{Karlsson2023}. 
Their approaches, however, do not seem to be 
adaptable for general commuting 
families of nonexpansive mappings. 
\end{remark}

\section{Main results}

Our main results are stated next and their 
proofs are given in Section~\ref{sec:proofs}.
For a mapping $T:X\to X$ we denote by $m(T,X)$
the nonnegative real number 
$\inf\{\norm{x-Tx}\;\mid\;x\in X\}$.

\begin{theorem}\label{thm:1}
Let $E$ be a Banach space and let $\F$ be a 
commuting family of affine nonexpansive 
mappings from $E$ to itself. If $X$ is a
nonempty convex subset of $E$ with the 
property that for all $T\in\F$ we have 
$TX \subset X$ and $m(T,X) = 0$, then
there exists a metric functional 
$\,\h\in X^\diamondsuit$ such that for all 
$x\in E$ and for all $T\in\F$ we have
$$
    \h(Tx) \leq \h(x).      
$$
In other words, $\h$ is a subinvariant metric
functional for all $T\in\F$.
\end{theorem}

\begin{remark}\label{rem:1}
Kohlberg and Neyman \cite{Kohlberg_Neyman1981} 
showed that every nonexpansive mapping 
$T:X \to X$ defined on a convex subset $X$ of 
a Banach space has the following property: 
for all $y \in X$ we have 
$$
\lim_{n\to\infty}\frac{\norm{T^{n}y}}{n}
    = \inf\{\norm{x-Tx}\;\mid\;x\in X\}.
$$
Thus, the assumption $m(T,X) = 0$ in 
Theorem~\ref{thm:1} holds whenever there is a 
vector $y\in X$ such that for all $T\in\F\,$ 
the limit shown above equals zero. This 
clearly holds for example when $X$ is bounded.  
\end{remark}

The following statement holds true for 
nonexpansive mappings, assuming neither 
affinity nor convexity, with the condition
that there is a vector $x_0$ such that 
\begin{equation}\label{melkeintarkei}
    \lim_{n\to\infty} 
    \norm{T^n x_0 - T^{n+1} x_0} = 0.
\end{equation}

\begin{theorem}\label{thm:2}
Let $E$ be a Banach space and let $\F$ be a 
commuting family of nonexpansive mappings 
from $E$ to itself. If $X$ is a nonempty 
subset of $E$ such that $T X \subset X$ for 
all $T\in\F$ and has a vector $x_0$ such that 
\textup{(\ref{melkeintarkei})} holds for all 
$T \in \F$, then there exists a metric 
functional $\,\h\in X^\diamondsuit$ that is 
subinvariant for all $\,T\in\F$.
\end{theorem}

We emphasize that our results are valid 
without any compactness assumption as opposed 
to what is assumed in classical fixed-point 
theorems such as those obtained by Markov and 
Kakutani \cite{Markov36, Kakutani38}, 
DeMarr \cite{DeMarr1963}, and 
Browder \cite{Browder1965}. As it was 
indicated in Proposition~\ref{proppu1}, for 
every nonexpansive mapping $T$, the existence 
of internal metric functionals that are 
subinvariant for $T$ is equivalent to the 
existence of fixed points of $T$. Internal 
metric functionals appear naturally when $T$ 
maps a compact set to itself. 

Sometimes we may be able to conclude the 
existence of a fixed point even when the 
subinvariant metric functional is not 
internal. This happens for example in cases 
where one knows all the metric functionals on 
the considered space. Guti\'{e}rrez 
\cite{Gutierrez2019-2} showed explicit 
formulas for all the metric functionals on 
the $\ell_p$ spaces with $1 \leq p < \infty$, 
we recall those formulas in Section~\ref{sec:2}.  
Having these available, the existence of a 
common fixed point follows from 
Theorem~\ref{thm:1} as shown below.

\begin{corollary}\label{cor:1}  
Assume that $1\leq p<\infty$. Suppose that 
$\F$ is a commuting family of affine 
nonexpansive mappings from $\ell_{p}$ to 
itself. If there exists a nonempty bounded 
subset $B$ of $\,\ell_{p}$ such that 
for all $\,T\in\F$ we have $TB\subset B$, then 
the family $\F$ has a common fixed point
in $\ell_p$. More precisely, there exists a 
vector $z\in\ell_{p}$ such that for all 
$\,T\in\F\,$ we have $\,Tz=z$.
\end{corollary}

\subsection{A fixed point theorem}

The fixed point $z$ shown in 
Corollary~\ref{cor:1} is not necessarily a 
point in the set $B$. We need additional 
properties if we insist on looking for a 
fixed point inside the given set. A key 
property that holds in all the $\ell_p$ 
spaces with $1\leq p<\infty$ is 
\emph{the Opial property} \cite{Opial1967}. 
We recall this property in 
Definition~\ref{Opial}.  
 
The following statement could serve as a 
model result to obtain a fixed point from 
a subinvariant metric functional.

\begin{theorem}\label{thm:3} 
Let $X$ be a nonempty weakly compact subset 
of a Banach space that has the Opial property. 
If a mapping $T:X\to X$ has a 
subinvariant metric functional 
then $T$ has a fixed point in $X$.
\end{theorem}
 
\section{About metric functionals}\label{sec:2}

Metric functionals on Banach spaces were 
investigated for example in \cite{Walsh2007}, 
\cite{JS2017}, \cite{Walsh2018},
\cite{Gutierrez2019-1}, \cite{Gutierrez2019-2}, 
\cite{Gutierrez2020}, \cite{LP2023}, 
\cite{CKS2023}. 
We recall next explicit formulas for 
all metric functionals on the $\ell_p$ spaces 
with $1\leq p <\infty$, where the norm is
$$
\norm{x}_p=\Big(\sum_{k\geq 1}\abs{x_k}^p\Big)^{1/p}\ \ 
\text{for all}\ \ x=(x_k)_{k\geq 1}\in \ell_p.
$$

\begin{theorem} 
\textup{(\cite[Section~5]{Gutierrez2019-2})}
\label{AG1}
Assume that $1<p<\infty$. Then, each metric
functional $\,\h\in (\ell_{p})^\diamondsuit$ 
is either a continuous linear functional with 
norm at most $1$, or a functional of the form
\begin{equation}\label{eq:mf_lp}
\h(x) = \big(\norm{x - z}_p^p + c^p 
    -\norm{z}_p^p\big)^{1/p}-\,c 
\end{equation}
for some $z\in\ell_p$ 
and some $c\in\R$ with $c\geq \norm{z}_{p}$.

Moreover, if $X$ is a bounded subset of 
$\ell_p$ then each metric functional 
$\h\in X^\diamondsuit$ is of the form 
\textup{(\ref{eq:mf_lp})}.
\end{theorem}

\begin{theorem}
\textup{(\cite[Section~3]{Gutierrez2019-2})}
We have $\,\h \in (\ell_{1})^\diamondsuit$ if
and only if there exists a subset $I$ of 
$\,\N$, an element $\varepsilon$ of 
$\,\{-1,1\}^I$, and an element $z$ of 
$\,\R^{\N\setminus I}$ such that for all 
$x \in \ell_1$ we have
\begin{equation} \label{eq:mf_l1}
\h(x) = \sum_{i\in I}\varepsilon_ix_i
    + \sum_{i\not\in I}(\abs{x_i-z_i}-\abs{z_i}).
\end{equation}

Moreover, if $X$ is a bounded subset of 
$\ell_{1}$ then each metric functional
$\,\h \in X^\diamondsuit$ is of the form 
$\,\h(x) = \norm{x - w}_{1} - \norm{w}_{1}\,$
for some $w\in\ell_1$.
\end{theorem}

Now we proceed to show the claims associated 
with the mapping $T$ introduced in 
Example \ref{ex:1}. 

\begin{proposition}\label{prop:l1-l2claims}
For the mapping $T:E\to E$ defined in \textup{Example~\ref{ex:1}} the following 
holds:
\begin{enumerate}
\item If $E=\ell_2$ then the only 
subinvariant metric functional for $T$ 
vanishes identically.
\item If $E=\ell_1$ then there are infinitely 
many subinvariant metric functionals for $T$, 
all of them nontrivial. 
\end{enumerate}
\end{proposition}

\begin{proof}
    To show the first statement we use Theorem~\ref{AG1} as 
    follows. First, assume that any metric functional of the 
    form (\ref{eq:mf_lp}) is subinvariant for $T$. This 
    implies that there exists a vector $z\in \ell_2$ such 
    that $\norm{Tx-z}_2 \leq \norm{x-z}_2$ for all 
    $x\in \ell_2$, which is a contradiction as $T$ has no 
    fixed points in $\ell_2$. Thus, the remaining candidates 
    must be linear functionals $\h(x) = \dpair{x}{z}$ where
    $z\in\ell_2$ with $\norm{z}_2\leq 1$. Now, $\h$ is 
    subinvariant for $T$ if and only if $\dpair{x-Tx}{z} \geq 0$ 
    for all $x \in \ell_2$.  Denote $z= (z_1,z_2, \dots)$.  
    If $z_1\neq 0$, let $t=(z_1-1)/z_1$. Since $z\in\ell_2$,
    there exists an integer $n>1$ such that 
    $\abs{z_n} \leq 1/(2\abs{t}+1)$. Define a vector 
    $x=(x_1,x_2, \dots)\in \ell_2$ by $x_j=t$ for $j<n$ 
    and $x_j=0$ for $j\geq n$. Then $(x-Tx,z) \leq -1/2$, 
    and hence the functional $\h$ is not subinvariant. 
    Finally, if $m>1$ is the smallest index $j$ such that 
    $z_j \neq 0$, then consider $x=(x_1,x_2, \dots)\in \ell_2$ 
    such that $x_{m-1} = 1/z_m$ and $x_j=0$ for $j\neq m-1$. 
    Then $(x-Tx,z) = -1$.

    To show the second statement note first that the 
    $0-$functional is not a metric functional on $\ell_1$, 
    see (\ref{eq:mf_l1}). Now, for each integer $N\geq 1$ 
    let  
    $$
        \h^{(N)}(x) = \sum_{j=N+1}^\infty (-x_j) 
            + \sum_{j=1}^N (|x_j-1| - 1)\ \ 
            \text{for all}\ \ x\in\ell_1.
    $$
    Then $\h^{(N)}(x) - \h^{(N)}(Tx) = \abs{x_N -1} 
        + x_N \geq 1$.  
\end{proof}

Notice that if $T$ is nonexpansive and has a unique 
fixed point $z$, then the internal metric functional 
$\h_z$ need not be the only subinvariant functional 
for $T$. In fact, let $S$ be the forward shift in 
$\ell_2$, then with all $c\geq 0$ the functional 
$\h(x)= ( \norm{x}^2 + c^2)^{1/2} -\;c$ is invariant 
for $S$, see (\ref{eq:mf_lp}).

\subsection{Properties of metric functionals}

As we observed previously, there are 
metric spaces which have the metric functional 
that vanishes identically. It is therefore 
reasonable to determine conditions under which
a metric space has only nontrivial metric
functionals. To this end, we introduce
here some relevant concepts.
 
\begin{definition} 
A metric space $X$ has 
\emph{the zero-free property}, ZFP, if for 
each metric functional 
$\h \in X^\diamondsuit$ there is 
a point $x \in X$ such that $\h(x)\neq 0$.
\end{definition}

\begin{definition}
A metric space $X$ has  
\emph{the unique minimizer property}, UMP, 
if for each metric functional 
$\h \in X^\diamondsuit$ there is a point 
$a \in X$ such that for all 
$x\in X\setminus\{a\}$ we have 
$\h(a) < \h(x)$.
\end{definition}

The following implication is immediate.

\begin{proposition}    
If a metric space has UMP then it has 
ZFP.
\end{proposition}
  
We emphasize that ZFP and UMP depend truly on 
the metric. To show this, let us assume that 
$(A,d_A)$ and $(B,d_B)$ are metric spaces and 
let $X_p$ denote the direct sum $A\oplus B$   
equipped with different metrics $d_p$:
$$
d_p = \begin{cases}
        (d_A^p + d_B^p)^{1/p} \ \ 
        \text{for}\ \ 1\leq p < \infty, \\
        \max\{ d_A, d_B\}\ \ 
        \text{for}\ \ p = \infty.
    \end{cases}
$$

\begin{proposition}  
For the metric space $X_1$ the following 
properties hold:
\begin{enumerate}
\item We have 
$\,\h \in (X_1)^\diamondsuit\,$ if and only 
if there are two metric functionals 
$\,\h^A\in A^\diamondsuit\,$ and 
$\,\h^B\in B^\diamondsuit$ such that
$$
\h(x) = \h^A(a) + \h^B(b)\ \ 
\text{for all}\ \ x=(a,b)\in X_1. 
$$
\item The metric space $X_1$ has ZFP if and 
only if at least one of $A$ and $B$ has ZFP.
\item The metric space $X_1$ has UMP if and 
only if both $A$ and $B$ have UMP.
\end{enumerate}
\end{proposition}

\begin{proof}
Let us fix two points $a_0\in A$ and 
$b_0\in B$. 

The first statement follows from the 
formula
$$d_1((a,b),(w,z)) - d_1((a_0,b_0),(w,z)) 
    = \h_{w}^{A}(a)+ \h_{z}^{B}(b),$$ 
where $\h_{w}^{A}(a)=d_A(a,w)-d_A(a_0,w)$ 
and $\h_{z}^{B}(b)=d_B(b,z)-d_B(b_0,z)$.     

Let us consider the second statement. 
If both $A$ and $B$ have identically 
vanishing metric functionals, their sum 
also vanishes. Now, let us assume that $A$ 
has ZFP and consider a metric functional 
$\h = \h^A + \h^B$. We know that there exists
$\hat{a} \in A$ such that 
$\h^A(\hat{a})\neq 0$. Since $\h^B(b_0)=0$, 
we have 
$\h(\hat{a},b_0) = \h^A(\hat{a}) \neq 0$.
If $B$ has ZFP instead, then we have
$\h(a_0,\hat{b})\neq 0$ for some 
$\hat{b}\in B$.    

To show the last statement, let us assume 
first that both $A$ and $B$ have UMP. Then, 
for every metric functional $\h = \h^A + \h^B$
there are $\hat{a}\in A$ and $\hat{b}\in B$
such that $\h^A(\hat{a}) < \h^A(a)$ for all 
$a\neq \hat{a}$ and $\h^B(\hat{b}) < \h^B(b)$ 
for all $b\neq \hat{b}$. This implies that
$\h(\hat{a},\hat{b}) < \h(a,b)$ for all 
$(a,b) \neq (\hat{a},\hat{b})$. Now, let us 
assume that $X_1$ has UMP and consider two 
metric functionals $\h^A\in A^\diamondsuit$ 
and $\h^B\in B^\diamondsuit$. Since $X_1$ has
UMP, there are two points $\hat{a}\in A$ and 
$\hat{b}\in B$ such that 
$$
\h^A(\hat{a})+\h^B(\hat{b}) 
< \h^A(a)+\h^B(b)\ \ \text{for all}\ \ 
(a, b)\neq (\hat{a}, \hat{b}). 
$$
In particular, by evaluating the inequality
shown above at the point $(a,\hat{b})$ with 
$a\neq \hat{a}$, we have 
$\h^A(\hat{a}) < \h^A(a)$. If we consider
the point $(\hat{a},b)$ with $b\neq \hat{b}$ 
instead, we have $\h^B(\hat{b}) < \h^B(b)$.
\end{proof}

For $1<p\leq \infty$ the situation is quite 
different:

\begin{proposition} 
Consider now the metric space $X_p$ where 
$1<p\leq \infty$. Assume that $B$ is 
unbounded and has a metric functional 
$\,\h^B= \lim_n \h_n^B$ where 
$\,\h_n^B(b) = d_B(b, b_n) - d_B(b_0, b_n)$ 
with $d_B(b_0, b_n) \to \infty$. Then the 
mapping
$$(a,b) \mapsto \h^B(b)$$ 
is a metric functional on $X_p$.
\end{proposition}

\begin{proof}  
Let us fix $a_0\in A$ and consider internal 
metric functionals on $X_p$ of the form
$$
\h_{(a_0,b_n)}(a,b) = d_p((a,b),(a_0,b_n))
            - d_p((a_0,b_0),(a_0,b_n)).
$$  
Then for the case $1<p<\infty$ we have 
$$
\h_{(a_0, b_n)}(a,b) = \h^B(b) + o(1)\ \ 
    \text{as}\ \ d_B(b_0, b_n) \to \infty,
$$
and for the case $p=\infty$ we have
$$
\h_{(a_0, b_n)}(a,b) = 
\max\{d_A(a,a_0)-d_B(b_0,b_n),\,\h_n^B(b)\}.
$$
From the formulas shown above we conclude 
that $\h_{(a_0,b_n)}(a,b) \to \h^B(b)$ as 
$d_B(b_0, b_n) \to \infty$.
\end{proof}
 
\begin{example}  
Consider the space 
$X_p = \ell_1 \oplus \ell_2$, where $\ell_1$ 
and $\ell_2$ are equipped with their standard 
norms. Then, $X_1$ has ZFP while $X_p$ with 
$1 < p \leq \infty$ does not.
\end{example}

We will hereinafter consider properties
of metric functionals on subsets $X$ of 
Banach spaces $E$. The standard notations
$B_E$ and $S_E$ will denote respectively the 
closed unit ball of $E$ and the unit sphere 
of $E$.

\begin{proposition}
If a Banach space $E$ has ZFP then 
$X=\delta B_E$ has ZFP for all positive
real numbers $\delta$.
\end{proposition} 
 
\begin{proof}  
This follows from the compactness of 
$E^\diamondsuit$ and the scaling property of 
the linear space $E$. 
Concretely, let us assume that 
$X=\delta B_E$ does not have ZFP for some 
$\delta >0$. Then there exists a metric 
functional $\h\in X^\diamondsuit$ that 
vanishes identically in $X$. We know that 
there is a net $(a_\alpha)$ in $X$ such 
that $\h(v)=\lim_\alpha \h_\alpha(v)$ where
$\h_\alpha(v) = \norm{v-a_\alpha}
-\norm{a_\alpha}$ for all $v\in E$.  
Now, for each integer $m\geq 1$ let us 
consider the internal metric functionals 
$$
\h_\alpha^{m}(v) = \norm{v- m a_\alpha} 
    - \norm{m a_\alpha} 
    = m \h_\alpha(m^{-1} v).
$$
For fixed $m\geq 1$ and $v\in E$, the 
limit $\lim_{\alpha} \h_\alpha^{m}(v)$ exists
and equals $m\h(m^{-1} v)$. By compactness of
$E^\diamondsuit$, the aforementioned limit 
determines a metric functional 
$\h^m\in E^\diamondsuit$ that vanishes 
identically in the ball $mX$. Thus, for all
$v\in E$ we have $\h^m(v)\to 0$ as 
$m\to\infty$. This limit determines precisely 
the metric functional vanishing identically 
in the whole $E$, because $E^\diamondsuit$ is 
compact.
\end{proof}

A simple modification of the previous proof 
gives the following. 

\begin{proposition} 
Let us assume that $X$ is a cone of a 
Banach space $E$, meaning $tx \in X$ for all 
$x\in X$ and for all $t>0$. If there exists a 
metric functional $\,\h \in X^\diamondsuit\,$ 
vanishing identically in $X \cap \delta B_E$ 
for some $\delta>0$, then $X$ does not have 
ZFP.
\end{proposition}

Before we state more properties of metric 
functionals, let us recall some standard 
notations used in Banach space theory.
For a given Banach space $E$ we will denote 
by $E^*$ the set of all continuous linear 
functionals on $E$. We know that $E^*$
becomes a Banach space itself when we equip it
with the operator norm. To simplify our 
exposition, we will use the same notation to
denote both the norm on $E$ and the operator 
norm on $E^*$. For a nonzero vector 
$x \in E$ we denote by $\partial \norm{x}$ 
the set of subdifferentials:
$$
    \partial \norm{x} = \{ f \in E^* 
    \mid \dpair{x}{f} 
    = \norm{x},\ \norm{f} = 1 \}
$$ 
and by $j(x)$ the set of dual vectors:
$$
    j(x) = \{ f \in E^* \mid \dpair{x}{f} 
        = \norm{x}^2,\ \norm{f} 
        = \norm{x} \}.
$$

We show below that every subset of Banach 
space containing a ray has a metric 
functional without lower bounds. In this 
context, a metric functional associated with 
a ray is sometimes called a 
\emph{Busemann function}.

\begin{proposition}\label{prop:Busemann} 
Let $X$ be a subset of a Banach space $E$ 
such that there exists a vector $u\in S_E$ 
with the property that $tu\in X$ for all 
$t \geq 0$. Then, for all $x\in E$ we have
$$
\lim_{t\to \infty} (\norm{x- tu} - t) 
    = \max \{ -\dpair{x}{f} \mid f \in 
    \partial \norm{u} \}.
$$
This limit determines a metric functional 
$\,\h\in X^\diamondsuit$ such that 
$\,\h(su) = -s$ for all $s\geq 0$.
\end{proposition}

\begin{proof}  
 
This follows immediately from the definition 
of the set $\partial \norm{u}$.  In fact, if
we fix $x\in E$ then for all 
$f\in \partial \norm{u}$ and for all $t>0$
we have
\begin{equation}\label{subd}
\norm{t u - x} \ge \norm{t u} + \dpair{-x}{f}.
\end{equation} 
The limit exists as it takes place in a 
two-dimensional subspace spanned by $u$ and 
$x$. We notice also that the real-valued 
function $t\mapsto \norm{x - tu} - \norm{t u}$ 
is non-increasing. As (\ref{subd}) holds for 
all $f$ in the norm-closed set 
$\partial\norm{u}$, the maximum is obtained 
in the limit. 
\end{proof}

\begin{example}  
Let us consider 
$X=\{n e_n\}_{n\in \mathbb Z}$ as a subset of 
the Banach space $c_0$, both equipped with the
sup-norm. Thus, $X$ is unbounded but does not 
contain a ray. We notice that each metric 
functional $\h \in X^\diamondsuit$ is either 
internal or the zero functional.
\end{example}

\begin{remark}
If the unit sphere $S_E$ is smooth, the set 
$\partial \norm{u}$ contains only one element, 
say $u^*$. The corresponding metric functional 
obtained in Proposition~\ref{prop:Busemann} 
becomes 
$$
\h(x) = - \dpair{x}{u^*}.
$$
\end{remark}

\begin{example} \label{1normi} 
Let us equip $\mathbb R^2$ with the norm 
$\norm{x} = \abs{x_1} + \abs{x_2}$ and fix 
the point $u=(1, 0 )^t$ to determine the ray 
so that 
$\partial\norm{u} = \{(1,\eta) \mid  |\eta | \le 1 \}$ 
and 
$\max_{|\eta |\le 1}(-x_1- \eta x_2) = -x_1 +  |x_2|.$ 
If we perturb the unit ball a little so that 
it becomes uniformly convex but still has 
corners with a little bit larger angle, then 
with some $0<\alpha <1$, we obtain likewise 
$\h(x)= - x_1 +\alpha |x_2|$.   
\end{example} 

The following result was shown in 
\cite[Section~5]{Gutierrez2019-2}. We present 
here a different proof. 

\begin{theorem}  
Let $E$ be a uniformly smooth Banach space. 
If $\,\h\in E^\diamondsuit$ is a metric 
functional that arises from a net 
$(a_\alpha)$ in $E$ when 
$\norm{a_\alpha} \to \infty$,
then there exists $u^* \in B_{E^*}$ such that 
$\,\h(x)= \dpair{x}{u^*}$ for all $x\in E$.

Conversely, if $u^* \in B_{E^*}$ is given 
then there exists a sequence $(a_n)$ in $E$ 
such that $\,\h_{a_n}(x) \to \dpair{x}{u^*}\,$ 
for all $x \in E$.  
\end{theorem}  

\begin{proof}  
Since $E$ is uniformly smooth, the norm of $E$ has the 
    following property \cite[p.~36]{Diestel1975}: 
    $$
        \lim_{t\to 0}\frac{\norm{u+tv} -\norm{u} 
            - \dpair{tv}{j(u)}}{\norm{tv}} = 0
    $$
    uniformly for all $u,v\in S_E$, where $j(u)$ is the unique 
    dual vector of $u$. Assume that $a_\alpha\neq 0$ for all 
    $\alpha$ and let $u_\alpha = a_\alpha / \norm{a_\alpha}$.
    Fix $x\in E$ and let $t v = -x/\norm{a_\alpha}$. 
    Note that
    $$
        \h_\alpha(x) =  \norm{x-a_\alpha} - \norm{a_\alpha} 
                    = \norm{a_\alpha}(\norm{u_\alpha + tv}-1).
    $$
    Thus, as $\norm{a_\alpha}\to\infty$ we get
    $$
        \h_\alpha(x) = \norm{a_\alpha}\big(\dpair{tv}{j(u_\alpha)} 
                        + \norm{tv}o(1)\big)  
                = - \dpair{x}{j(u_\alpha)} + \norm{x}o(1).
    $$
    By the Banach-Alaoglu theorem, the net $(j(u_\alpha))$ has a limit point
    $-u^*\in B_{E^*}$, and hence $\h(x)= \dpair{x}{u^*}$.
    
    Conversely, let $u^* \in B_{E^*}$, choose any sequence $(f_n^*)$ 
    in $S_{E^*}$ converging weakly to $0$ and scalars $t_n$ so that 
    $$ 
        \norm{u^* + t_n f_n^*} = 1.
    $$   
    Let $u_n\in S_E$ be such that $j(u_n)= u^* + t_nf_n^*$ and 
    let $a_n = nu_n$. Then $\h_{a_n}(x)$ converges to 
    $\dpair{x}{u^*}$ for all $x\in E$. 
\end{proof}

As seen from the Example \ref{1normi}, the nonsmooth points on the 
unit sphere $S_E$ create metric functionals which are not linear 
and simultaneously all linear functionals of norm at most $1$
are not metric functionals. 

In $\ell_1$, every metric functional that is linear has the form 
$\h(x) =  \sum_{j=1} ^\infty  \varepsilon_j x_j$ where 
$\varepsilon_j \in \{ -1, 1 \}$ for all $j\in \N$. Thus, $\ell_1$ 
does have ZFP.    

Karlsson showed \cite[Prop.~22]{Karlsson2023} that $\ell_\infty$ 
has ZFP. Here we give a different proof. Let $S$ be a nonempty 
set and let $\mathcal{B}(S)$ be the space of real-valued bounded 
functions equipped with sup-norm.
\begin{proposition}
    $\mathcal{B}(S)$ has ZFP.
\end{proposition}

\begin{proof}
Consider $\h\in \mathcal{B}(S)^\diamondsuit$ 
and assume that $(a_\alpha)$ is a net in 
$\mathcal{B}(S)$ such that 
$\h=\lim_\alpha \h_\alpha$ where 
$\h_\alpha(x)=\norm{x-a_\alpha}-\norm{a_\alpha}$ 
for all $x\in \mathcal{B}(S)$. Now, let 
$(\delta_\alpha)$ be a net of positive reals 
converging to $0$. Then, for each $\alpha$, 
there exists a point $s_\alpha\in S$ such that
$$
\norm{a_\alpha} \leq \abs{a_\alpha(s_\alpha)} 
            + \delta_\alpha.
$$
Suppose that $(a_\alpha)$ has a subnet 
$(a_\beta)$ for which $a_\beta(s_\beta) > 0$. 
Let $u\in \mathcal{B}(S)$ be the constant 
function $u(s)=-1$ for all $s\in S$. It follows 
that
$$
\norm{u-a_\beta} \geq 1 + a_\beta(s_\beta) 
        \geq 1 + \norm{a_\beta} -\delta_\beta.
$$
Thus we have $1 - \delta_\beta \leq \h_\beta(u) \leq 1$, 
and hence $\h(u)=1$. If such subnet $(a_\beta)$ 
does not exist, consider $u=1$ instead.
\end{proof}

By an argument quite similar to the one 
used in the previous proof, we can show that
the Banach space $\mathcal{C}(K)$ of 
continuous functions on a compact Hausdorff
space $K$ has ZFP.

\begin{proposition}
Let $K$ be a compact Hausdorff space and let 
$u(s)=1$ for all $s\in K$. Then for all 
$\,\h \in \mathcal{C}(K)^\diamondsuit$ and 
for all $t \geq 0$ we have
$$
        \h(tu) = t\ \ \text{or}\ \ \h(-tu) = t.
$$ 
In particular, $\mathcal{C}(K)$ has ZFP.
\end{proposition}

As any normed space $E$ has metric functionals which are not 
bounded below, the question of $X\subset E$ having UMP is 
interesting essentially only when $X$ is bounded. We shall 
formulate the observations for $X=B_E$, the closed unit ball 
of $E$. 

The first observation is that $B_{c_0}$ does not have ZFP. 
Hence, additional assumptions are needed on $E$ such that $B_E$ 
would have ZFP or UMP. We consider three different properties 
of $E$ which guarantee that $B_E$ has UMP or at least ZFP.   

\begin{definition}
A Banach space $E$ has the Radon-Riesz property if weak 
convergence and convergence of the norms imply strong 
convergence: $a_n \rightarrow a$ weakly and 
$\norm{a_n} \to \norm{a}$ imply $\norm{a_n - a} \to 0$.
\end{definition}

For example, every uniformly convex Banach space has the 
Radon-Riesz property.

\begin{proposition}
    Let $E$ be a Banach space with the Radon-Riesz property. 
    If the closed unit ball $B_E$ is weakly compact, then 
    $B_E$ has ZFP.
\end{proposition}

\begin{proof} 
Let $\h$ be an element of $(B_E)^\diamondsuit$. 
Let us assume that $(a_\alpha)$ is a net of 
vectors in $B_E$ such that for all $x \in E$ 
we have $\h_\alpha(x) \to \h(x)$ where 
$\h_\alpha(x) = 
\norm{x-a_\alpha}-\norm{a_\alpha}$. 
Since $(\norm{a_\alpha})$ is a net of real 
numbers in the closed interval $[0,1]$, there 
is a subsequence $(a_{\alpha_n})$ such that 
$\norm{a_{\alpha_n}} \to r$ where 
$\norm{a} \leq r \leq 1$.
As $B_E$ is weakly compact, the 
Eberlein-\v{S}mulian theorem \cite{Diestel1984}
implies the existence of a subsequence, denoted 
again by $(a_{\alpha_n})$, that converges 
weakly to some $a \in B_E$. 

Let us assume that $\norm{a}$ is positive 
and consider the vector 
$y= -\frac{1}{\norm{a}}a$ in $B_E$. Then, we 
have $\norm{y-a} = 1 + \norm{a}$. If we 
denote by $f$ the linear functional 
$\frac{1}{1+\norm{a}} j(y-a)$ in $B_{E^*}$, 
we have 
$$
\norm{y - a_{\alpha_n}} -\norm{a_{\alpha_n}}  
    \geq  \dpair{y - a_{\alpha_n}}{f} 
    - \norm{a_{\alpha_n}}.        
$$
From the inequality shown above we get  
$\h(y) \geq \norm{y-a} - r = 1 + \norm{a} - r > 0$.
 
Now, let us assume that $\norm{a}$ equals $0$. 
We may assume that $r$ is positive, otherwise
we have $\h(x) = \norm{x}$ for all $x$. Let 
us choose a vector $z$ such that 
$\norm{z} = r$ and assume that 
$\norm{z - a_{\alpha_n}} \rightarrow r$. The 
Radon-Riesz property implies that 
$z - a_{\alpha_n}$ converges strongly to $z$, 
that is, $a_{\alpha_n}$ converges strongly to 
$0$, which is a contradiction. Thus, we 
have $\h(z) \neq 0$.
\end{proof}

\begin{definition}\label{Opial} 
A Banach space $E$ has the Opial property, 
if whenever a sequence $(a_n)$ in $E$ 
converges weakly to $a$, the following holds: 
$$
\liminf \norm{a_n-a} < \liminf \norm{a_n - x}\ \ 
        \text{for all}\ \ x \neq a.
$$
\end{definition}

It is known that all the $\ell_p$ spaces with 
$1 \leq p < \infty$ have the Opial property. 
It is shown in 
\cite[Theorem~1]{DalbySims1996} that a Banach 
space has the Opial property if and only if 
whenever $(a_n)$ converges weakly to a 
nonzero limit $a$, the number 
$\liminf_n \dpair{a}{a_n^*}$ is positive, 
where $a_n^* \in j(a_n)$. 

\begin{proposition}\label{prop:WeakCompUMP}
Let $E$ be a Banach space with the Opial 
property. If $X$ is a nonempty weakly compact 
subset of $E$, then $X$ has UMP. 
\end{proposition}

\begin{proof}
Let $\h$ be an element of $X^\diamondsuit$. 
Let us assume that $(a_\alpha)$ is a net in 
$X$ such that for all $x \in X$ we have
$\h(x) = \lim_\alpha\h_{\alpha}(x)$, where
$\h_\alpha(x) = \norm{x-a_\alpha} 
    - \norm{a_\alpha}$.
As $X$ is weakly compact, there is a subsequence 
$(a_{\alpha_n})$, a vector $a\in X$, and a 
nonnegative real number $r$ such that 
$a_{\alpha_n}\to a$ weakly and 
$\norm{a_{\alpha_n}}\to r$.   
Due to the Opial property, for all 
$x \in X \setminus \{a\}$ we have
$$
\h(x) = \lim_{n} \norm{x - a_{\alpha_n}} - r 
    \,>\, \lim_{n} \norm{a - a_{\alpha_n}} - r 
    \,=\, \h(a).
$$
\end{proof}

Another important concept is the following.

\begin{definition}
(\cite{Lim1977}, \cite[Definition~3.1]{Solimini-Tintarev2015})  
Let $(a_n)$ be a sequence in a Banach space 
$E$. If there exists a vector $a \in E$ such 
that for all $x \in E$ we have 
$$
    \norm{a_n -a} \leq \norm{a_n - x} + o(1)
$$
when $n\to \infty$, then one says that 
$(a_n)$ is $\Delta-$convergent to $a$. 
\end{definition}

Weak convergence and $\Delta$-convergence 
coincide on Banach spaces that have the Opial 
property and are both uniformly smooth and 
uniformly convex 
\cite[Theorem~3.19]{Solimini-Tintarev2015}.
This is true in particular on all Hilbert 
spaces and all the $\ell_p$ spaces with 
$1<p<\infty$.
A key result for $\Delta$-convergence is the 
following theorem, which has a resemblance to
the Banach-Alaoglu theorem.

\begin{theorem}
\textup{(\cite[Theorem~4]{Lim1977})}
\label{BADelta}  
Let $E$ be a uniformly convex Banach space. 
Then, every bounded sequence $(a_n)$ in $E$ 
has a $\Delta$-convergent subsequence.
\end{theorem}

We have immediately the following.

\begin{proposition}  
Let $X$ be a bounded subset of a uniformly 
convex Banach space $E$. If $(a_n)$ 
is a sequence in $X$ such that 
$(\h_{a_n})$ converges to a metric functional
$\,\h\in X^\diamondsuit$, then 
there exists a unique vector $a\in E$ such that 
for all $x\in E$ we have $\h(a) \leq \h(x)$.
\end{proposition}

\begin{proof}
Due to Theorem~\ref{BADelta}, the sequence 
$(a_n)$ has a subsequence $(a_{n_k})$ that is 
$\Delta$-convergent to some vector $a\in E$.
By taking a suitable subsequence, we may 
assume that $\norm{a_{n_k}} \to r$. Then 
for all $x \in E$ the limit 
$\lim_{k} \norm{x - a_{n_k}}$ exists and 
equals $\h(x) + r$. By $\Delta$-convergence, 
we have $\h(a) \leq \h(x)$ for all $x\in E$.
Now, let us assume that there
exists $b\neq a$ such that $\h(a)=\h(b)$.
Since $\h$ is a convex functional, we have
$\h(\frac{1}{2}(a+b)) = \h(a)$.
Thus, we have 
$$
\lim_{k} \norm{a - a_{n_k}} 
    = \lim_{k} \norm{b - a_{n_k}}
    = \lim_{k} \norm{\frac{1}{2}(a+b) - a_{n_k}}
    = \h(a) + r.
$$
We cannot have $\h(a) + r = 0$ because
we assumed that $b\neq a$. So, let us assume 
that $M =\h(a) + r > 0$.
Since $E$ is uniformly convex, we must have
$\lim_k\norm{M^{-1}(a - a_{n_k}) - M^{-1}(b - a_{n_k})}=0$.
Thus, we have $M^{-1}\norm{a-b}=0$, 
which is a contradiction.
\end{proof} 

\section{Proofs of the main results}
\label{sec:proofs}

\begin{proof}
[\bf{Proof of Theorem~{\upshape\ref{thm:1}}}]
Let $T$ and $U$ be two elements of $\F$. That 
is, both $T$ and $U$ are affine nonexpansive
mappings from $E$ to itself and for all 
$x \in E$ we have $TUx=UTx$. Let us choose an
arbitrary vector $w\in X$. For each positive
integer $n$ let us consider the vectors $a_n$ 
and $b_n$ defined by the formulas 
$a_n = n^{-1}(w + Tw + \cdots + T^{n-1}w)$
and
$b_n = n^{-1}(a_n + Ua_n + \cdots + U^{n-1}a_n)$.
Since both $T$ and $U$ map the convex set $X$
into itself, both $(a_n)$ and $(b_n)$ are 
sequences in $X$ such that
$$
\norm{b_n - T b_n} 
    \leq \norm{a_n - T a_n} 
    \leq n^{-1}\norm{w - T^n w} 
$$
and 
$$
\norm{b_n - U b_n} 
    \leq n^{-1}\norm{a_n - U^n a_n} 
    \leq n^{-1}\norm{w - U^n w}.
$$
Due to Remark~\ref{rem:1}, our assumption 
$m(T,X)=m(U,X)=0$ implies that  
$\norm{b_n - T b_n} \to 0\,$ and 
$\norm{b_n - U b_n} \to 0\,$ 
when $n\to\infty$. Now, for each $n\geq 1$ 
let $\h_n \in E^\diamondsuit$ denote the 
metric functional defined for all $x\in E$
by the formula 
$\h_n(x) = \norm{x-b_{n}}-\norm{b_{n}}$.
Next, we notice that for all $x\in E$ we have
$$
\h_n(Tx) - \h_n(x) \leq \norm{b_n - T b_n}
$$
and 
$$
\h_n(Ux) - \h_n(x) \leq \norm{b_n - U b_n}.
$$
The compactness of $E^\diamondsuit$ implies 
that $(\h_n)$ has a limit point 
$\h\in E^\diamondsuit$ such that 
$\h(Tx)\leq \h(x)$ and $\h(Ux)\leq\h(x)$.

We have so far proved our theorem for two 
elements of the family $\F$. Our claim is in 
fact true for all finite subsets of the 
family $\F$, as a quick inspection of the 
previous procedure reveals. With that fact in 
mind, we now proceed to prove the general 
case. For each $T \in \F\,$ let $M_T$ denote 
the set of all metric functionals 
$\h \in E^\diamondsuit$ such that 
$\h(Tx)\leq \h(x)$ for all $x \in E$. We
notice that each $M_T$ is a nonempty closed 
subset of $E^\diamondsuit$ and the family 
$\{M_T\,\mid\,T \in \F\}$ has the finite 
intersection property. Since $E^\diamondsuit$
is compact, the set
$$
\bigcap_{T \in \F} M_T 
$$
is nonempty, and hence it contains a metric 
functional $\h$ that has the desired 
property.
\end{proof}

\begin{proof}
[\bf{Proof of Corollary~{\upshape\ref{cor:1}}}]
Let $E=\ell_p$ and let $X$ be the convex hull 
of $B$. Thus, $X$ is a nonempty bounded 
convex subset of $E$ such that $TX \subset X$ 
for all $T \in \F$. It follows from 
Remark~\ref{rem:1} that for all $T\in\F\,$ we 
have $m(T,X) = 0$. By Theorem~\ref{thm:1},  
there exists a metric functional 
$\h\in X^\diamondsuit$ that
is subinvariant for all $T\in \F$. Since $X$ 
is bounded, the candidates for $\h$ are of 
the form (\ref{eq:mf_lp}) when $p>1$ or 
internal when $p=1$. Thus, there exists a 
vector $z\in E$ such that for all $x\in E$ and 
for all $T\in \F$ we have
$$
\norm{Tx-z}_p \leq \norm{x-z}_p
$$
Therefore, $z\in E$ is the common fixed 
point of the family $\F$. 
\end{proof}

\begin{proof} 
[\bf{Proof of Theorem~{\upshape\ref{thm:2}}}]
Let us assume that $T$ and $U$ are two 
elements of the family $\F$. For each 
positive integer $n$ let $x_n$ denote the 
vector $T^{n}U^{n}x_0$. As both $T$ and 
$U$ map $X$ into $X$ and $x_0\in X$, 
$(x_{n})$ is a sequence in $X$. Now, let 
$(\h_n)$ be the sequence in 
$E^\diamondsuit$ defined by the formula
$\h_n(x)=\norm{x - x_{n}}-\norm{x_{n}}$ for 
all $x\in E$. Since $T$ and $U$ are defined
on the whole space $E$, for all $x\in E$ and
for all $n\geq 1$ we have
$$ 
\h_n(Tx) - \h_n(x) 
    \leq \norm{x_{n} - Tx_{n}}
    \leq \norm{T^{n}x_0 - T^{n+1}x_0} 
$$
and 
$$ 
\h_n(Ux) - \h_n(x) 
    \leq \norm{x_{n} - Ux_{n}}
    \leq \norm{U^{n}x_0 - U^{n+1}x_0}. 
$$
In the two previous inequalities each term 
situated farthest to the right is assumed to
converge to $0$ when $n\to\infty$. Therefore,
the compactness of $E^\diamondsuit$ implies 
that $(\h_n)$ has a limit point 
$\h\in E^\diamondsuit$ such that 
$\h(Tx)\leq \h(x)$ and $\h(Ux)\leq\h(x)$. 
The rest of the proof follows from a 
compactness argument similar to the one used 
in the proof of Theorem~\ref{thm:1}.
\end{proof}

\begin{proof} 
[\bf{Proof of Theorem~{\upshape\ref{thm:3}}}]
Let us assume that $\h\in X^\diamondsuit$ is 
subinvariant for the mapping $T:X\to X$. That 
is, for all $x\in X$ we have 
$\h(Tx) \leq \h(x)$. We know that the space 
$X$ has UMP due to 
Proposition~\ref{prop:WeakCompUMP}. Thus, 
there is a vector $a\in X$ such that for all 
$x\in X\setminus\{a\}$ we have 
$\h(a) < \h(x)$. The preceding two 
inequalities imply that $Ta=a$.
\end{proof}

\section{Acknowledgements}
Armando W. Guti\'{e}rrez was supported by the 
Research Council of Finland 
(funding decision number 348093).
\bibliographystyle{amsplain}
\bibliography{subinvariantMF}

\end{document}